\documentclass[12pt]{article}
\usepackage[margin=.85in]{geometry}
\normalsize

\usepackage[T1]{fontenc}
\usepackage[utf8]{inputenc}
\usepackage{amssymb,amsmath,amsthm,bm,mathrsfs,paralist}
\usepackage{url,xcolor,units,graphicx,mathtools,bbold}
\RequirePackage[colorlinks,citecolor=blue,urlcolor=blue,linkcolor=blue]{hyperref}
\usepackage{pdfsync}
\usepackage{cite}

\usepackage{diagbox}

\newcommand{\R}{\mathbb{R}}

\newcommand{\E}{\mathbb{E}}
\renewcommand{\P}{\mathbb{P}}

\let\leq=\leqslant
\let\geq=\geqslant
\let\le=\leqslant

\newtheorem{proposition}{Proposition}
\newtheorem{theorem}[proposition]{Theorem}
\newtheorem{lemma}[proposition]{Lemma}
\newtheorem{corollary}[proposition]{Corollary}
\newtheorem{definition}[proposition]{Definition}

\newtheorem{example}[proposition]{Example}

\theoremstyle{definition}
\newtheorem{remark}[proposition]{Remark}

\numberwithin{equation}{section}
\numberwithin{proposition}{section}

\newcommand{\bR}[1]{\left(#1\right)}

\def\dfrac#1#2{\lower0.12ex\hbox{\large$\textstyle\frac{#1}{#2}$}}

\makeatletter
\def\and{%
  \end{tabular}%
  \begin{tabular}[t]{c}}
\makeatother

\usepackage{enumitem}
\setlist[itemize]{noitemsep, nolistsep}
\setlist[enumerate]{noitemsep, nolistsep}

\title{
{\bf\Large Multivariate Poisson approximation 
of joint subgraph counts in random graphs via size-biased couplings}
}

\author{Eulalia Nualart\\
\small Universitat Pompeu Fabra \\ [-0.5ex]
\small Barcelona School of Economics \\ [-0.5ex]
\small Ram\'on Trias Fargas 25-27, Barcelona 08005, Spain \\ [-0.5ex]
\small\tt eulalia.nualart@upf.edu
\and
\hspace{-.6cm}Rui-Ray Zhang \\
\hspace{-.6cm}\small Barcelona School of Economics \\ [-0.5ex]
\hspace{-.6cm}\small\tt rui.zhang@bse.eu}
\date{}

\begin{document}

\maketitle

\begin{abstract}

Using Chen-Stein method in combination with size-biased couplings,
we obtain the multivariate Poisson approximation
in terms of the  Wasserstein distance.
As applications, 
we study the multivariate Poisson approximation of the distribution of joint subgraph counts in an Erd\H{o}s-Rényi random graph
and the multivariate hypergeometric distribution
giving explicit convergence rates.

\vskip 12pt

	
\end{abstract}

%\tableofcontents

\section{Introduction}

Let $\mathcal G(n, p)$ denote  the Erd\H{o}s-Rényi random graph  on $n$ vertices with edge
probability $p=p(n) \in (0, 1)$.
Let $H$ be a strictly balanced graph with $v_H>0$ vertices and $e_H$ edges,
and let $W$ denote the number of copies of  $H$ in $\mathcal G(n, p)$.
It is well-known since the 80's (see e.g. Janson, \L{}uczak, and Ruci\'nski \cite[Theorem 3.19]{J00}) that using the method of moments, 
if $np^{d_H} \rightarrow c>0$ as $n \rightarrow \infty$, 
where $d_H = e_H/v_H$ denotes the density of $H$,
then $W$ converges in distribution to a Poisson random variable $P_{\lambda}$ with parameter $\lambda=c^{v_H}/a_H$, 
where $a_H$ denotes the size of the automorphism group of $H$. 
Furthermore, using the Chen-Stein method, one can obtain an  explicit rate of convergence 
in terms of the total variation distance of both distributions, see Barbour, Holst, and Janson \cite[Chapter 5]{BHJ} or \cite[Example 6.26]{J00}. More specifically, if $H$ has no isolated vertices, then as $n \rightarrow \infty$,
$$
d_{TV}(W, P_{\lambda})=O(n^{-\gamma}),
$$
where $
\gamma 
= \min\left\{v_{H'} -  e_{H'}/d_{H} 
: H' \subsetneq H, e_{H'} > 0\right\} >0
$
and $\lambda=\E(W) \rightarrow c^{v_H}/a_H$.
This is based on the notion of size-biased coupling that we recall in Section 2. See Ross \cite[Section 4]{R} for an extended survey  of this topic.

One of the goals of this paper is to extend this rate of convergence to the joint convergence of several subgraph counts in terms of the Wasserstein distance, which is stronger than the total variation one. More specifically,  
let $(H_1,\ldots,H_d)$ be a sequence of $d$ distinct strictly balanced graphs with $0 < v_{H_i} \leq n$ vertices, $e_{H_i}>0$ edges,  same density $\alpha=d_{H_i}$, for $i=1,\ldots, d$, and no isolated vertices. Let ${\bf W}=(W_1, \ldots, W_d)$ , where for $i=1,\ldots,d$, $W_i$ denotes the number of copies of $H_i$ in $\mathcal G(n,p)$. Then, we show in Theorem \ref{t5b}, that if for some constant $c>0$,
$
p = cn^{-1/\alpha}(1 + o(1))$, then as $n \rightarrow \infty$,
\[
d_{\mathrm{W}}(\mathbf{W}, {\bf P_{\boldsymbol{\lambda}}})=O(n^{-\gamma}),
\]
where ${\bf P}_{{\bm \lambda}}=(P_1, \ldots, P_d)$ denotes a Poisson random vector 
with independent components and expectation ${\bm \lambda}=(\lambda_1, \ldots, \lambda_d)$
with $\lambda_i =\E(W_i)\rightarrow c^{\alpha v_{H_i}}/a_{H_i}$,
and
$$
\gamma 
= \min_{i\in \{1,\ldots,d\}} \min\{v_{H'_i} - e_{H'_i}/\alpha : H'_i \subsetneq H_i, e_{H'_i} > 0\}>0.
$$
If the graphs do not have the same density, but there is a critical common density $\alpha$ shared by a subset of them, our Theorem \ref{t5b} also studies the behavior of $W_i$ for the graphs that do not have density equal to $\alpha$.
In the case that the graphs are not necessarily strictly balanced and $p \rightarrow 0$ as $n \rightarrow \infty$, we obtain a bound for the Wasserstein distance of both distributions that can be applied to a wide class of examples in order to show multivariate Poisson convergence, see Corollary \ref{t5} and Remark \ref{re}.

In order to achieve this first aim, we use the multivariate Poisson approximation developed by Pianoforte and Turin \cite{PT}, which differs from the multivariate extension given by Goldstein and Rinott \cite{GR} and that by Arratia, Goldstein, and Gordon \cite{AGG}. In  \cite{PT}, the authors use Stein's equation for the Poisson distribution iteratively in order to derive a multivariate Poisson approximation bound in terms of the Wasserstein distance for general integer-valued random  vectors. As applications, they obtain the multivariate Poisson approximation of dependent Bernoulli sums and Poisson process approximation of point processes  of $U$-statistic structure. As  the authors observe in  their paper, a particular case of  their bound is when one can find random vectors with an exact size-biased distribution in a multivariate  setting that  differs from previous multivariate extensions  such as \cite{GR}.  In our paper, we exploit this fact by formalizing the multivariate notion of size-biased coupling and studying the case  of multivariate sums of indicators random variables. Then, as a consequence of the multivariate Poisson approximation  theorem obtained in \cite{PT}, we obtain explicit bounds in terms of the Wasserstein distance for increasing and decreasing multivariate size-biased couplings, extending the one-dimensional case (see for e.g. \cite{R}). Although the multivariate Poisson distribution has independent components, our bounds take into account of  the  dependences of  the random vector through its covariances. We then  apply the increasing size-biased multivariate Poisson approximation bound to the joint convergence of several subgraph counts mentioned above. This needs a 
careful analysis of the covariance structure of subgraph counts, extending  the covariance study done in Krokowski and Th\"ale \cite{KT}. In \cite[Theorem 4.2]{KT}, the  rate of convergence to the multivariate Gaussian distribution of joint subgraph counts is achieved  in the case that $p$ is constant, using the fourth differentiable function distance and discrete Malliavin calculus.
Finally, as an application of  the decreasing size-biased multivariate Poisson approximation, we consider the multivariate hypergeomertic distribution.

While we focused on Erd\H{o}s-Rényi random graphs, 
the techniques developed here may potentially be extended to other random graph models,
such as random geometric graphs, etc.
The multivariate size-biased coupling framework could also be useful in other contexts
such as pattern occurrences in random permutations, and local statistics in random discrete structures.
It may be possible to obtain high-order asymptotic probability with smaller error,
by considering clusters of subgraphs, as in Zhang \cite[Chapter 5]{zhang}.

The paper is organized as follows. In Section 2 we develop the general theory of multivariate  size-biased coupling and Poisson approximations.
Then in Section 3 we consider  the two applications, namely the distribution of  subgraph counts  in  random graphs and the multivariate hypergeometric  distribution.

{\bf Notation:}
We use the notation $\mathbb{N}=\{1,2, \ldots \}$ and $\mathbb{N}_0=\{0,1,2, \ldots \}$.
Throughout this paper, 
${\bf P}_{{\bm \lambda}}=(P_1, \ldots, P_d)$ denotes a Poisson distributed random vector 
with independent components and expectation ${\bm \lambda}=(\lambda_1, \ldots, \lambda_d)$.

\section{Size-biased coupling and Poisson approximation}

Size-biased couplings first appeared in the context of Stein's method for
normal approximation by Goldstein and Rinott \cite{GR},
and they are also useful when used in conjunction with Stein's method to obtain Poisson approximation.
Following the survey by Ross \cite{R}, we recall the notion of size-biased coupling.
\begin{definition}[Size-biased coupling]
Let $W  \geq 0$ be a random variable with $\E(W)=\lambda<\infty$. We say  that a  random  variable $\widetilde W$ is a size-biased coupling of $W$ if for all function $f$ satisfying  $\E\vert  Wf(W)\vert  < \infty$, we have
$$
\E(Wf(W))=\lambda \E(f(\widetilde W)).
$$
\end{definition}

Observe that if $W\geq 0$ is  an integer-valued random variable with finite mean $\lambda>0$, then  $\widetilde W$ is a size-biased coupling of   $W$ if and only if for all $k \in \mathbb{N}_0$, 
$$
\P(\widetilde W=k)=\frac{k }{\lambda} \P(W=k).
$$

Applying size-biased couplings gives Poisson approximation for real-valued random variables.
\begin{theorem}{\cite[Theorem 4.13]{R}}
Let $W\geq 0$ be  an integer-valued random variable with $\E(W)=\lambda>0$, 
and let $\widetilde W$ be a size-biased coupling of $W$. 
Then, if $P_{\lambda}$ is a Poisson distributed random variable with parameter $\lambda$, we have
$$
d_{\mathrm{TV}}(W, P_{\lambda}) \leq \min\{1,\lambda\} \E \vert  \widetilde W-1-W \vert,
$$
where the total variation (TV) distance between two distributions $P$ and $Q$ on some finite domain $\mathcal{D}$
is defined by
$$
d_{\mathrm{TV}}(P, Q)
=\dfrac{1}{2} \sum_{x \in \mathcal{D}}|P(x)-Q(x)|
=\max_{S \subseteq \mathcal{D}} | P(S)-Q(S) |.
$$
\end{theorem}

\subsection{Multivariate size-biased coupling}
We extend the size-biased coupling to the multivariate case,
following the result by Pianoforte and Turin \cite{PT}.
\begin{definition}[Multivariate size-biased coupling]
Let ${\bf W}=(W_1, \ldots, W_d)$ be a random vector with $\E(W_i)=\lambda_i < \infty$, for $i=1, \ldots, d$.
Let
${\bf  \widetilde W}=({\bf \widetilde W}^{1}, \ldots, {\bf \widetilde W}^{d})$ be a $d$-dimensional triangular array such that 
${\bf \widetilde W}^{i}=(\widetilde W_1^{i},  \ldots, \widetilde W^{i}_i)$ is an $i$-dimensional random vector for each $i=1, \ldots, d$.
We say that ${\bf \widetilde W}$ is a size-biased coupling of ${\bf W}$ 
if 
\begin{equation} \label{h}
\E(W_i f({\bf W}^i))=\lambda_i \E(f({\bf \widetilde W}^{i})), \quad 
\text{ with }
{\bf W}^i=(W_1, \ldots, W_i)
\end{equation}
for all $i=1, \ldots, d$, and all functions $f:\R^i \rightarrow \R$ satisfying $\E \vert W_i f({\bf W}^i)\vert <\infty$.
\end{definition}

Observe that if ${\bf W}=(W_1, \ldots, W_d)$ takes values in $\mathbb{N}^d_0$, 
$0<\E(W_i)=\lambda_i <\infty$,
and the random vector ${\bf \widetilde W}^{i}$ takes values in $\mathbb{N}^i_0$ for all $i=1, \ldots, d$, then condition (\ref{h}) is equivalent of saying that for all $i=1, \ldots, d$ and ${\bf k}^i=(k_1, \ldots, k_i) \in \mathbb{N}_0^i$, 
we have
\begin{equation} \label{h2}
\P({\bf \widetilde W}^{i}={\bf k}^i)=\frac{k_i }{\lambda_i} \P({\bf W}^i={\bf k}^i).
\end{equation}

The following extends \cite[Corollary 4.14]{R} to the multivariate setting. 
\begin{lemma} \label{c1b}
For all $i=1, \ldots, d$, let $(X^i_1, \ldots, X^i_{n_i})$ be an $n_i$-dimensional random vector formed by indicator random variables
such that for $j=1, \ldots, n_i$,
$
\P(X_j^i=1)=p_{i,j} \in (0,1)$, and set ${\bf W}=(W_1, \ldots, W_d)$,
with $
W_i=\sum_{j=1}^{n_i}  X_j^i$.
For all  $i=1, \ldots, d$ and $\ell=1, \ldots, n_i$, consider 
an $n_1  \cdots n_{i-1} (n_i-1)$-dimensional random vector 
$$
{\bf X}^{i,(i,\ell)}
= \bR{ 
(X_{j}^{1,(i,\ell)})_{j=1}^{n_1}, \ldots, (X_{j}^{i-1,(i,\ell)})_{j=1}^{n_{i-1}},(X_{j}^{i,(i,\ell)})_{j=1,j\neq \ell}^{n_i} 
}
$$
defined on the same probability space as the random vector
$$
{\bf X}^{i,\ell}
= \bR{ 
(X_{j}^1)_{j=1}^{n_1}, \ldots, (X_{j}^{i-1})_{j=1}^{n_{i-1}},(X_{j}^{i})_{j=1, j \neq \ell}^{n_i}
}
$$
satisfying
\begin{align}\label{law}
\mathcal{L}\,({\bf X}^{i,(i,\ell)})=\mathcal{L}\,({\bf X}^{i,\ell} \mid X^i_{\ell}=1).
\end{align}
Let $I_1, \ldots, I_d$ be a sequence of independent random variables, independent of all else, 
such that for all $i=1, \ldots, d$ and $j=1, \ldots, n_i$,
\begin{align}\label{Idef}
\P\left( I_i=j\right)=\frac{p_{i,j}}{\lambda_i},
\qquad \text{ where  }
\lambda_i=\E(W_i)=\sum_{j=1}^{n_i} p_{i,j}.
\end{align}
Let ${\bf  \widetilde W}=({\bf \widetilde W}^{1}, \ldots, {\bf \widetilde W}^{d})$ be the $d$-dimensional triangular array given by, for $i=1, \ldots, d$,
\begin{align}\label{tilWdef}
{\bf \widetilde W}^{i}
= \bR{ \sum_{j=1}^{n_1} X_{j}^{1,(i,I_i)}, \ldots, \sum_{j=1}^{n_{i-1}} X_{j}^{i-1,(i,I_i)},
\sum_{j=1, j \neq I_i}^{n_i} X_{j}^{i,(i,I_i)}+1 }.
\end{align}
Then, ${\bf  \widetilde W}$ is a size-biased coupling of ${\bf W}$.
\end{lemma}
\begin{proof}
It suffices to prove (\ref{h2}). For $i=1, \ldots, d$ and ${\bf k}^i=(k_1, \ldots, k_i) \in \mathbb{N}_0^i$, in view of \eqref{Idef} and \eqref{tilWdef},
we have, conditioning  with respect to all the  values of $I_i$ and  using independence together with the conditional probability law equality \eqref{law}, that
\begin{equation*} \begin{split} 
&\lambda_i\P({\bf \widetilde W}^{i}={\bf k}^i) 
=\lambda_i \sum_{j=1}^{n_i}  \P({\bf \widetilde W}^{i}={\bf k}^i, I_i=j) \\
&=\sum_{j=1}^{n_i} \P((W_1, \ldots, W_{i-1},W_i-X_{j}^i+1)={\bf k}^i \mid X^i_{j}=1) p_{i,j}\\
&=\sum_{j=1}^{n_i} \frac{p_{i,j}}{\P(X^i_{j}=1)}\P({\bf W}^i={\bf k}^i, X^i_{j}=1) 
=\sum_{j=1}^{n_i} \P({\bf W}^i={\bf k}^i, X^i_{j}=1)
=k_i \P({\bf W}^i={\bf k}^i),
\end{split}
\end{equation*}
recalling the definition of $p_{i,j}$ in \eqref{Idef}.
This concludes the desired proof.
\end{proof}

The following extends \cite[Corollary 4.15]{R} to the multivariate setting. 
\begin{corollary} \label{c}
For all $i=1, \ldots, d$, let $(X^i_1, \ldots, X^i_{n_i})$ be an $n_i$-dimensional random vector formed by exchangeable indicator random variables and consider 
an $n_1  \cdots n_{i-1} (n_i-1)$-dimensional random vector 
$${\bf X}^{i,(i,1)}
= \bR{
(X_{j}^{1,(i,1)})_{j=1}^{n_1}, \ldots, (X_{j}^{i-1,(i,1)})_{j=1}^{n_{i-1}},(X_{j}^{i,(i,1)})_{j=2}^{n_i}
}$$
defined on the same probability space as the random vector
$${\bf X}^{i,1}
= \bR{ (X_{j}^1)_{j=1}^{n_1}, \ldots, (X_{j}^{i-1})_{j=1}^{n_{i-1}},(X_{j}^{i})_{j=2}^{n_i} }
$$
satisfying
$$
\mathcal{L}\,({\bf X}^{i,(i,1)})=\mathcal{L}\,({\bf X}^{i,1}  \mid X^i_{1}=1).
$$
 Let ${\bf  \widetilde W}=({\bf \widetilde W}^{1}, \ldots, {\bf \widetilde W}^{d})$ be the $d$-dimensional triangular array given by, for $i=1, \ldots, d$,
$$
{\bf \widetilde W}^{i}
= \bR{ \sum_{j=1}^{n_1} X_{j}^{1,(i,1)}, \ldots, \sum_{j=1}^{n_{i-1}} X_{j}^{i-1,(i,1)},
\sum_{j=2}^{n_i} X_{j}^{i,(i,1)}+1 }.
$$
Then, ${\bf  \widetilde W}$ is a size-biased coupling of ${\bf W}$.
\end{corollary}
\begin{proof}
The proof follows from Lemma \ref{c1b} and the fact that exchangeability implies that the random variables $I_i$'s are uniformly distributed.
\end{proof}

\subsection{Multivariate Poisson approximation}

Recall that the Wasserstein distance between random vectors $\mathbf{X}$ and $\mathbf{P}$
is defined by
$$
d_{\mathrm{W}}(\mathbf{X}, \mathbf{P})=\sup _{g \in \operatorname{Lip}^d(1)}
|\mathbb{E}( g(\mathbf{X}) ) - \mathbb{E}( g(\mathbf{P}) )|,
$$
where $\operatorname{Lip}^d(1)$ denotes the set of Lipschitz functions $g: \mathbb{N}_0^d \rightarrow \mathbb{R}$ with Lipschitz constant bounded by 1 with respect to the 1-norm $|\mathbf{x}|_1=\sum_{i=1}^d\left|x_i\right|$, 
for a vector $\mathbf{x}=\left(x_1, \ldots, x_d\right) \in \mathbb{R}^d$.
Since the indicator functions defined on $\mathbb{N}_0^d$ are Lipschitz, 
for random vectors in $\mathbb{N}_0^d$, 
the Wasserstein distance dominates the total variation distance,
and therefore, all bounds henceforth are also valid for total variation distance.

Our first result is the following Poisson approximation
using multivariate size-biased coupling, which is a consequence of the result by Pianoforte and Turin \cite{PT}.

\begin{theorem} \label{t1}
Let ${\bf W}=(W_1, \ldots, W_d)$ be a $d$-dimensional random vector taking values in $\mathbb{N}^d_0$, with $\E(W_i)=\lambda_i>0$ for $i=1, \ldots, d$
and let ${\bf  \widetilde W}$ be a size-biased coupling of ${\bf W}$. Then,
$$
d_{\mathrm{W}}({\bf W}, {\bf  P_{\bm\lambda}}) \leq \sum_{i=1}^d \min\{1,\lambda_i\} 
\E \vert \widetilde W_i^{i}-1-W_i\vert
+2 \sum_{i=2}^d  \lambda_i \sum_{j=1}^{i-1} \E\vert \widetilde W_j^{i}-W_j \vert.
$$
\end{theorem}

\begin{proof}
The proof is a direct application of \cite[Theorem 1.1]{PT} 
by taking, for $i=1, \ldots, d$, $Z_i^{(i)}=\widetilde W_i^{i}-1-W_i$,
$Z_j^{(i)}=\widetilde W_j^{i}-W_j$ for $j=1, \ldots, i-1$, and ${\bf Y}^{(i)}={\bf \widetilde W}^{i}$. 
Then, we observe that (\ref{h2}) implies that all the $q^{(i)}_{m_{1:i}}$'s in Theorem 1.1 in \cite{PT} are zero, 
since in our notation, we have that for all $i=1, \ldots, d$ 
and ${\bf k}^i=(k_1, \ldots, k_i) \in \mathbb{N}_0^i$ with $k_i \neq 0$,
$$
q^{(i)}_{m_{1:i}}=k_i \P({\bf W}^i={\bf k}^i)-\lambda_i\P({\bf \widetilde W}^{i}={\bf k}^i),
$$
which concludes the proof.
\end{proof}

\begin{remark}
We observe that \cite[Theorem 1.1]{PT} is stated with $\min\{1,\lambda_i\}$ replaced by $\lambda_i$. However, when going through the proof, 
we see that we can use a sharper bound 
%for the increment of the Stein's function 
given for e.g. in \cite[Lemma 4.4]{R}, which allows to replace the $\lambda_i$ by $\min\{1,\lambda_i\}$.
\end{remark}

When the multi-dimensional size-biased coupling
is increasing, we extend \cite[Theorem 4.20]{R}.
\begin{theorem} \label{i1}
Under the assumptions of Lemma \ref{c1b}, suppose that for all $i=1, \ldots, d$, $j=1, \ldots, i-1$, $\ell=1, \ldots, n_i$, and $k=1, \ldots, n_j$, we have  
$X_{k}^{j,(i,\ell)}\geq  X_{k}^{j}$, 
and for $j=i$ and $k \neq \ell$ we have $X_{k}^{i,(i,\ell)}\geq  X_{k}^{j}$ . Then, 
$$
d_{\mathrm{W}}({\bf W}, {\bf  P_{\bm\lambda}}) \leq\sum_{i=1}^d \min\{1,\lambda_i^{-1}\} \left(\textnormal{Var}(W_i)-\lambda_i+2\sum_{j=1}^{n_i} p_{i,j}^2\right)
+2 \sum_{i=2}^d  \lambda_i^{-1} \sum_{j=1}^{i-1} \textnormal{Cov}( W_i,W_j).
$$
\end{theorem}

\begin{proof}
Applying Theorem \ref{t1}, we have, by using \eqref{Idef}, that
\begin{equation*}
\begin{split}
&d_{\mathrm{W}}({\bf W}, {\bf  P_{\bm\lambda}}) \\
& \leq \sum_{i=1}^d \min\{1,\lambda_i^{-1}\} \sum_{j=1}^{n_i} p_{i,j} \E\left(\sum_{k=1, k \neq j}^{n_i}(X_{k}^{i,(i,j)}-
X_{k}^{i})+X_j^i\right) +
2 \sum_{i=2}^d  \lambda_i \sum_{j=1}^{i-1} \E\left( \widetilde W_j^{i}-W_j \right) \\
&=\sum_{i=1}^d \min\{1,\lambda_i \}\,\E \left( \widetilde W_i^{i}-W_i-1+2X_{I_i}^i\right) +2 \sum_{i=2}^d  \lambda_i \sum_{j=1}^{i-1} \E\left( \widetilde W_j^{i}-W_j \right).
\end{split}
\end{equation*}
Now, using (\ref{h}) and the fact that $\lambda_i=\E(W_i)$, we  get  that 
\[
\lambda_i\E ( \widetilde W_i^{i}-W_i)=\E(W_i^2)-(\E(W_i))^2=\textnormal{Var}(W_i),
\]
and similarly, 
\[
\lambda_i  \E( \widetilde W_j^{i}-W_j)=\E(W_i W_j)-\E(W_i) \E(W_j)=\textnormal{Cov}( W_i,W_j).
\]
Finally, using \eqref{Idef} to rewrite $\lambda_i$, we have
\[
\lambda_i \E(X_{I_i}^i)
= \lambda_i\sum_{j=1}^{n_i} \P(I_i = j)\E(X_{I_i}^i \mid I_i = j)
= \sum_{j=1}^{n_i} p_{i,j} \E(X_{I_i}^i \mid I_i = j)
= \sum_{j=1}^{n_i} p_{i,j}^2,
\]
where we use the definition of $p_{i,j}$.
This completes the desired  bound.
\end{proof}

Next, we extend \cite[Theorem 4.31]{R}
for the decreasing size-biased coupling.
\begin{theorem} \label{dd}
Under the assumptions of Lemma \ref{c1b}, suppose that for all $i=1, \ldots, d$, $j=1, \ldots, i-1$, $\ell=1, \ldots, n_i$, and $k=1, \ldots, n_j$, we have  
$X_{k}^{j,(i,\ell)}\leq  X_{k}^{j}$ and for $j=i$ and $k \neq \ell$ we have $X_{k}^{i,(i,\ell)}\leq  X_{k}^{j}$. Then, 
$$
d_{\mathrm{W}}({\bf W}, {\bf  P_{\bm\lambda}}) \leq \sum_{i=1}^d \min\{1,\lambda_i^{-1}\} \left(\lambda_i-\textnormal{Var}(W_i)\right)
-2 \sum_{i=2}^d  \lambda_i^{-1} \sum_{j=1}^{i-1} \textnormal{Cov}( W_i,W_j).
$$
\end{theorem}

\begin{proof}
Applying Theorem \ref{t1}, we have that
\begin{equation*}
\begin{split}
&d_{\mathrm{W}}({\bf W}, {\bf  P_{\bm\lambda}})\\
&\leq \sum_{i=1}^d \min\{1,\lambda_i^{-1}\}\sum_{j=1}^{n_i} p_{i,j} \E\left(\sum_{k=2}^{n_i}(X_{k}^{i}-
X_{k}^{i,(i,j)})+X_j^i\right)
+2 \sum_{i=2}^d  \lambda_i \sum_{j=1}^{i-1} \E\left( W_j- \widetilde W_j^{i} \right) \\
&=\sum_{i=1}^d \min\{1,\lambda_i\} \E \left(W_i- \widetilde W_i^{i}+1\right)
+2 \sum_{i=2}^d  \lambda_i \sum_{j=1}^{i-1} \E\left( W_j- \widetilde W_j^{i}\right).
\end{split}
\end{equation*}
Finally, proceedings as in the proof of Theorem \ref{i1}, we obtain the desired bound.
\end{proof}

\section{Applications}

\subsection{Joint subgraph counts in random  graphs}

Recall that $\mathcal G(n, p)$ denotes  the Erd\H{o}s-Rényi random graph  on $n$ vertices with edge
probability $p=p(n) \in (0, 1)$.
Let $(H_1, \ldots, H_d)$ be a sequence of distinct fixed graphs 
such that
for $i=1, \ldots, d$,
each graph $H_i$
has $0 < v_{H_i} \leq n$ vertices, $e_{H_i}>0$ edges, and no isolated vertices. 
We analyze the joint distribution of subgraph counts of  $(H_1, \ldots, H_d)$ in $\mathcal G(n, p)$, 
that is, the  number of copies of $H_i$ in $\mathcal G(n, p)$ for $i=1, \ldots, d$.

Our main result in this section is the multivariate Poisson approximation of 
subgraph counts, as a consequence of the decreasing size-biased coupling in Theorem \ref{i1}, 
as an extension of \cite[Theorem 4.21]{R}. 
Observe that in this application of Theorem \ref{i1} the  $p_{i,j}$'s only depend on $i$.
\begin{theorem} \label{t4}
Let $( H_1, \ldots, H_d )$ be a sequence of $d$ distinct graphs with $0 < v_{H_i} \leq n$ vertices, 
$e_{H_i}>0$ edges, and no isolated vertices. Let ${\bf W}=(W_1, \ldots, W_d)$,
where $W_i$ is the number of copies of
$H_i$ in $\mathcal G(n, p)$ for $i=1, \ldots, d$, and let $\lambda_i=\E(W_i)$.
Then,
$$
d_{\mathrm{W}}({\bf W}, {\bf  P_{\bm\lambda}}) \leq\sum_{i=1}^d \min\{1,\lambda_i^{-1}\} \left(\textnormal{Var}(W_i)-\lambda_i+2 \lambda_i p^{e_{H_i}}\right)
+2 \sum_{i=2}^d  \lambda_i^{-1} \sum_{j=1}^{i-1} \textnormal{Cov}( W_i,W_j).
$$
\end{theorem}

To obtain more explicit error bounds,
we introduce some notation to analyse the covariances between subgraph counts.
Let $H_i$ and $H_j$ be two graphs that can be isomorphic. 
For any fixed copies $\alpha_i \in \Gamma_i$ and $\beta_j \in \Gamma_j$ in the complete graph $K_n$, we denote by $e_{\alpha_i,\beta_j}$ the number of edges shared between $\alpha_i$ and $\beta_j$, and by $v_{\alpha_i,\beta_j}$ the number of vertices incident to these shared edges. Define
\[
M_{i,j} = \max\{e_{\alpha_i,\beta_j} : \alpha_i \in \Gamma_i, \beta_j \in \Gamma_j\},
\]
which represents the maximum possible number of shared edges between any copy of $H_i$ and $H_j$.
For each $k \in \{1,\ldots,M_{i,j}\}$, define 
\[
\ell_{k,i,j} = \begin{cases}
\min\{v_{\alpha_i,\beta_j} : e_{\alpha_i,\beta_j} = k, \alpha_i \in \Gamma_i, \beta_j \in \Gamma_j\} & \text{if } \mathcal{S}_k(i, j) \neq \emptyset,\\
0 & \text{if } \mathcal{S}_k(i, j) = \emptyset,
\end{cases}
\]
where 
$
\mathcal{S}_k(i, j) = \{(\alpha_i,\beta_j) \in \Gamma_i \times \Gamma_j : e_{\alpha_i,\beta_j} = k\}
$
is the set of pairs sharing exactly $k$ edges.  
When $\ell_{k,i,j} > 0$, it represents the minimum number of vertices needed to realize $k$ shared edges between copies of $H_i$ and $H_j$.
Note that
$
\binom{\ell_{k,i,j}}{2} \le k \le 2 \ell_{k,i,j},
$
and $\ell_{k,i,j} = 0$ when it is impossible
for $H_i$ and $H_j$ to share exactly $k$ edges.
Let 
\begin{align*}
\mathcal{K}_{i,j} = \{k \in \{1,\ldots,M_{i,j}\} : \ell_{k,i,j} > 0\}
\end{align*}
be the set of possible shared edge counts between $H_i$ and $H_j$.
Define
\begin{align*}
\gamma_i(t) 
= \min_{k \in \mathcal{K}_{i,i}} \left\{ \ell_{k,i,i} - \dfrac{k}{t} \right\}.
\end{align*}

As a consequence of Theorem \ref{t4}, we obtain an explicit error bound in the case that $p \rightarrow 0$ as $n \rightarrow \infty$ using the above notation.
\begin{corollary}\label{t5}
Under the same conditions as in Theorem \ref{t4}, assume that $p \to 0$ as $n \to \infty$. Then, there exists a constant $C > 0$ such that for $n$ sufficiently large, we have
\begin{align*}
d_{\mathrm{W}}(\mathbf{W}, \mathbf{P}_\lambda) 
\le C \left\{\sum_{i=1}^d \min\{1,\lambda_i\} \left(p^{e_{H_i}} + n^{v_{H_i}-\gamma_i}p^{e_{H_i}}\right) 
+ \sum_{i=2}^d \sum_{j=1}^{i-1} \lambda_j \sum_{k \in \mathcal{K}_{i,j}} p^{-k}n^{-\ell_{k,i,j}} \right\},
\end{align*}
where for $i=1,\ldots,d$,
\[
\gamma_i 
= \gamma_i(d_{H_i})
= \min\left\{v_{H'_i} -  \dfrac{e_{H'_i}}{d_{H_i}} 
: H'_i \subsetneq H_i, e_{H'_i} > 0\right\},
\]
and $d_{H_i}=e_{H_i}/v_{H_i}$
is the density of the graph $H_i$.
\end{corollary}

\subsubsection{Proofs of Theorem \ref{t4} and Corollary \ref{t5}}

\begin{proof}[Proof of Theorem \ref{t4}]
For $i=1, \ldots, d$, let $\Gamma_i$ denote the set of all copies  of $H_i$ in the complete graph of $n$ vertices $K_n$.
Let $X^i_{\alpha_i}$ the indicator that there is a copy of $H_i$ in $\mathcal G(n, p)$ at $\alpha_i$, 
and set $W_i=\sum_{\alpha_i \in \Gamma_i} X^i_{\alpha_i}$.
We apply Theorem \ref{i1} to the vector ${\bf W}$ by constructing a size-biased coupling of ${\bf W}$ using Corollary \ref{c} 
and the fact that each $W_i$ is a sum of exchangeable indicators. 

For all $i=1, \ldots, d$, $j=1, \ldots, i$, and $\alpha_i \in \Gamma_i$, 
let
$X_{\beta_j}^{j,(i,\alpha_i)}$ be the indicator 
that there is a copy of $H_i$ in $\mathcal G(n, p) \cup \{ \alpha_i\}$ at $\beta_j \in \Gamma_j$, 
where
$\mathcal G(n, p) \cup \{ \alpha_i\}$ denotes 
the graph obtained by
adding the minimum edges necessary to $\mathcal G(n, p)$ 
such that $\mathcal G(n, p) \cup \{ \alpha_i\}$ contains a copy of $H_i$ at $\alpha_i$.

Then the following facts imply the theorem:
\begin{itemize}
\item $
\mathcal{L}\,({\bf X}^{i,(i,\alpha_i)})=\mathcal{L}\,({\bf X}^{i,\alpha_i}  \mid X^i_{\alpha_i}=1),
$
where 
$$
{\bf X}^{i,(i,\alpha_i)}
= \bR{
(X_{\beta}^{1,(i,\alpha_i)})_{\beta \in \Gamma_1}, 
\ldots, (X_{\beta}^{i-1,(i,\alpha_i)})_{\beta \in \Gamma_{i-1}},
(X_{\beta}^{i,(i,\alpha_i)})_{\beta \in \Gamma_i, \beta \neq \alpha_i}
}
$$
and
$$
{\bf X}^{i,\alpha_i}
=\bR{
(X_{\beta}^1)_{\beta \in \Gamma_1}, \ldots, 
(X_{\beta}^{i-1})_{\beta \in \Gamma_{i-1}},(X_{\beta}^{i})_{\beta \in \Gamma_i, \beta \neq \alpha_i}
}.
$$

\item For all $i=1, \ldots, d$, $j=1, \ldots, i-1$, $\alpha_i \in \Gamma_i$, and $\beta_j \in \Gamma_j$, we have  
$X_{\beta_j}^{j,(i,\alpha_i)}\geq  X_{\beta_j}^{j}$,
and for $j=i$, $\beta_i \neq \alpha_i$, we have $X_{\beta_i}^{i,(i,\alpha_i)}\geq  X_{\beta_i}^{i}$,
by noting that we only add edges in the size-biased coupling construction.

\item $\E(X^i_{\alpha_i})= p^{e_{H_i}}$ for all $\alpha \in \Gamma_i$.
\end{itemize}
This completes the proof.
\end{proof}

Observe that applying the one-dimensional bound of \cite[Corollary 4.22]{R}, we obtain the  following bound for the first error term in Theorem \ref{t4}. For $i=1, \ldots, d$, $\alpha_i \in \Gamma_i$, and $H'_i \subseteq  H_i$, let $\Gamma_i^{H'_i, \alpha_i} \subseteq \Gamma_i$ be the set of subgraphs of $K_n$ isomorphic to $H_i$,
whose  intersection  with $\alpha_i$ is $H'_i$.
Then, for $i=1, \ldots, d$,
\begin{equation} \label{c4a}
\textnormal{Var}(W_i)-\lambda_i+2 \lambda_i p^{e_{H_i}} 
\leq
\lambda_i \bigg( p^{e_{H_i}}+\sum_{H'_i \subsetneq H_i: e_{H'_i} >0} \vert  \Gamma_{i}^{H_i', \alpha_i} \vert p^{e_{H_i}-e_{H_i'}}\bigg),
\end{equation}
where the terms $\vert  \Gamma_{i}^{H_i', \alpha_i} \vert$ account for  the number  of covariance terms for different types of pair indicators.

We next give a more explicit bound on the second error term in Theorem \ref{t4}. 
Recall that $a_H$ denotes the size of the set of automorphisms of graph $H$.
We observe that for $i=1, \ldots, d$,
\begin{equation*} 
\E(W_i)
= |\Gamma_i| p^{e_{H_i}}
=\binom{n}{v_{H_i}} \dfrac{v_{H_i}!}{a_{H_i}} p^{e_{H_i}}.
\end{equation*}
We have the following asymptotic result, 
which extends Lemma 4.1 in Krokowski and Th\"ale \cite{KT} to the case where $p$ depends on $n$,
and with errors therein corrected.
\begin{lemma} \label{lr}
Under the same conditions as in Corollary \ref{t5}, for all $i \in \{2, \ldots, d\}$ and $j \in \{1, \ldots, i-1\}$, we  have, for sufficiently large $n$,
\begin{equation*} \begin{split}
\frac{ \textnormal{Cov}( W_i,W_j)}{\E(W_i)\E(W_j)}
 = 
 O \bigg( \sum_{k \in \mathcal{K}_{i,j}} p^{-k} n^{-\ell_{k,i,j}}\bigg).
\end{split}
\end{equation*}
\end{lemma}

\begin{proof}
The  proof follows along the same lines as  in \cite[Lemma 4.1]{KT}, 
but taking into account the fact that $p$ depends on $n$. 
Let $i \in \{2, \ldots, d\}$ and $j \in \{1, \ldots, i-1\}$. Then
for sufficiently large $n$,
\begin{equation*}
\begin{split}
&\textnormal{Cov}( W_i,W_j) 
=\sum_{\alpha_i  \in \Gamma_i, \beta_j \in \Gamma_j\,:\, e_{\alpha_i,\beta_j}\geq 1} 
%\textnormal{Cov}(X^i_{\alpha_i},X^j_{\beta_j}) \\
%&=\sum_{\alpha_i  \in \Gamma_i, \beta_j \in \Gamma_j: e_{\alpha_i,\beta_j}\geq 1} 
\left(\E(X^i_{\alpha_i} X^j_{\beta_j})- \E(X^i_{\alpha_i}) \E( X^j_{\beta_j}) \right)\\
&=\sum_{\alpha_i  \in \Gamma_i, \beta_j \in \Gamma_j: e_{\alpha_i,\beta_j}\geq 1} \left(p^{e_{H_i}+e_{H_j}-e_{\alpha_i,\beta_j}}- p^{e_{H_i}} p^{e_{H_j}} \right)\\
&=\sum_{k \in \mathcal{K}_{i,j}}\sum_{(\alpha_i,\beta_j) \in \mathcal{S}_k(i, j)}  p^{e_{H_i}+e_{H_j}-k}(1-p^k)
\leq \sum_{k \in \mathcal{K}_{i,j}} \Delta_k,
\end{split}
\end{equation*}
where
$\Delta_k = p^{-k} \sum_{(\alpha_i,\beta_j) \in \mathcal{S}_k(i, j)}p^{e_{H_i}+e_{H_j}}$.

Next, we observe that  since the $v_{H_i}>0$ and  the graphs $H_i$ have no isolated vertices, we have that 
$1 \in  \mathcal{K}_{i,j}$. Let  $k=1$.
We have  $\binom{n}{v_{H_i}} \frac{v_{H_i}!}{a_{H_i}}=\frac{\E(W_i)}{p^{e_{H_i}}}$ possible choices for $\alpha_i \in \Gamma_i$.
Since $k=1$, given $\alpha_i \in \Gamma_i$, we need to compute the number of 
$\beta_j \in \Gamma_j$ that have exactly one edge in common with $\alpha_i$, by observing that we have $e_{H_i}$ possibilities to choose this common edge in $\alpha_i$.
Once the  common edge in $\alpha_i$ is fixed, we  have  $\binom{n-v_{H_i}}{v_{H_j}-2}$ possibilities to choose
the $v_{H_j}-2$ remaining vertices of $\beta_j$, $e_{H_j}$ possibilities to choose  the common  edge  in $\beta_j$, and $2(v_{H_j}-2)! / a_{H_j}$ distinct possible arrangements of the  vertices.
Then
\begin{equation*}
\Delta_1
= \E(W_i) e_{H_i}e_{H_j}
\binom{n-v_{H_i}}{v_{H_j}-2} \dfrac{2(v_{H_j}-2)!}{a_{H_j}}
p^{e_{H_j}-1} 
%&\leq  2e_{H_i}e_{H_j} \E(W_i) \E(W_j) n^{-2} p^{-1} \\
=\E(W_i)\E(W_j) O(n^{-2}p^{-1} ).
\end{equation*}

If $k=2$, that is, $\alpha_i$ and $\beta_j$ have exactly two common edges, then we have two possibilities. The first one is that  the  edges share  one common
vertex, thus, there are only three distinct vertices. In  this case, following as above we find that the contribution in the sum is $\E(W_i)\E(W_j)O(n^{-3}p^{-2})$. The second possibility is that the edges  do not share any vertex, that is, we have 4 distinct vertices. In this case the  contribution is $\E(W_i)\E(W_j)O(n^{-4}p^{-2})$.
Therefore, for $k\geq 2$, we observe that 
the terms that will contribute in the bound will be those
for which the $k$ edges in common have a  minimum number of distinct vertices, 
that is, $\ell_{k,i,j}$. 
This concludes the proof.
\end{proof}

%We next use inequality (\ref{c4a}) and Lemma \ref{lr} together with Theorem \ref{t4}, in order to obtain a general multivariate Wasserstein bound that can be applied to a wide class of examples.

\begin{proof}[Proof of Corollary \ref{t5}]
Applying the bound  (\ref{c4a}) and Lemma \ref{lr}, we get, from Theorem \ref{t4}, that
for sufficiently large $n$, the quantity $d_{\mathrm{W}}({\bf W}, {\bf  P_{\bm\lambda}})$ is bounded by
  \begin{equation*} 
  \begin{split}
&C \bigg\{\sum_{i=1}^d  \min\{1,\lambda_i\}
\bigg(p^{e_{H_i}}+\sum_{H'_i \subsetneq H_i: e_{H'_i} >0} \vert  \Gamma_{i}^{H_i', \alpha_i} \vert p^{e_{H_i}-e_{H_i'}}\bigg)
%\\ &\qquad \qquad \qquad  \qquad  
 +\sum_{i=2}^d  \sum_{j=1}^{i-1}\lambda_j\sum_{k \in \mathcal{K}_{i,j}} p^{-k} n^{-\ell_{k,i,j}}\bigg\} \\
&\leq C \bigg\{\sum_{i=1}^d  \min\{1,\lambda_i\}
\bigg(p^{e_{H_i}}+n^{v_{H_i}} p^{e_{H_i}}\sum_{H'_i \subsetneq H_i: e_{H'_i} >0} \left( n^{v_{H'_i}/e_{H'_i}-v_{H_i}/e_{H_i}}\right)^{-e_{H'_i}}\bigg)\\
&\qquad \qquad  \qquad  
+\sum_{i=2}^d \sum_{j=1}^{i-1} \lambda_j \sum_{k \in \mathcal{K}_{i,j}} p^{-k} n^{-\ell_{k,i,j}}\bigg\},
\end{split}
\end{equation*}
by noting that $\vert  \Gamma_{i}^{H_i', \alpha_i} \vert \le n^{v_{H_i} - v_{H'_i}}$.
This completes the proof.
\end{proof}

\subsubsection{Strictly balanced graphs}

We now apply Corollary \ref{t5} to a wide class of graphs 
called {\it strictly balanced} in order to obtain an exact convergence rate.
A graph $H$ is called 
\textit{strictly balanced} if $$d_{H'} < d_H  \text{ whenever } H' \subsetneq H,
\quad\text{where } d_H=\frac{e_H}{v_H}
$$ 
is the density of the graph $H$.
It is well-known that trees, cycles, and complete graphs are all strictly balanced graphs.
We have the following consequence Corollary \ref{t5} when all the graphs are strictly balanced.

\begin{theorem} \label{t5b}
Let $\mathcal G(n,p)$ be an Erdős-Rényi random graph with edge probability 
\[
p = cn^{-1/\alpha}(1 + o(1)),
\] 
where $c > 0$ and $\alpha > 0$ are fixed constants. Let $(H_1,\ldots, H_d)$  be a sequence of $d$
strictly balanced distinct graphs with $0 < v_{H_i} \leq n$ vertices, $e_{H_i} > 0$ edges, for $i=1,\ldots,d$, and no isolated vertices,  Let $W_i$ denote the number of copies of $H_i$ in $\mathcal G(n,p)$.
Define $I = \{i : d_{H_i} = \alpha\}$ as the set of critical indices. Then as $n \to \infty$,
we have the following.
\begin{enumerate}
\item For $i \notin I$: If $d_{H_i} < \alpha$ then
    \[
    \frac{W_i}{\mathbb{E}(W_i)} \xrightarrow{P} 1 \text{ at rate } n^{-\gamma_i},
    \]
    where 
\[
\gamma_i = \gamma_i(\alpha) = \min_{k \in \mathcal{K}_{i,i}} \left\{ \ell_{k,i,i} - \dfrac{k}{\alpha} \right\} >0.
\]
    If $d_{H_i} > \alpha$ then
    \[\mathbb{P}(W_i > 0) \leq \exp(-Cn^{\eta_i}),\]
    where 
    \[
    \eta_i = \eta_i(\alpha) = v_{H_i}(d_{H_i}/\alpha - 1) > 0.
    \]

\item For the critical subsequence $\mathbf{W}_I=(W_i)_{i \in I}$, we have, 
\[
d_{\mathrm{W}}(\mathbf{W}_I, {\bf P_{\boldsymbol{\lambda}_I}})\leq Cn^{-\gamma},
\]
where
$\boldsymbol{\lambda}_I = (\lambda_i)_{i\in I}$ with $\lambda_i =\E(W_i)\rightarrow c^{\alpha v_{H_i}}/a_{H_i}$
and
\begin{equation} \label{gamma}
\gamma 
= \min_{i\in I} \gamma_i(\alpha)
= \min_{i\in I} \min\{v_{H'_i} - e_{H'_i}/\alpha : H'_i \subsetneq H_i, e_{H'_i} > 0\}>0.
\end{equation}
\end{enumerate}
\end{theorem}

\begin{proof}
When $d_{H_i} < \alpha$, we use the second moment method. First note that
$
\mathbb{E}[W_i] \to \infty.
$
By Chebyshev's inequality and Lemma \ref{lr}, for all $\epsilon>0$,
\[
\mathbb{P}\left(\left|\frac{W_i}{\mathbb{E}[W_i]} - 1\right| \geq \epsilon\right) \leq \frac{\text{Var}(W_i)}{\epsilon^2\mathbb{E}[W_i]^2} 
= O \bigg( \sum_{k \in \mathcal{K}_{i,i}} n^{k/\alpha-\ell_{k,i,i}}\bigg)
= O\left(n^{-\gamma_i}\right),
\]
in view of the definition of $\gamma_i$.
By the strictly balanced property, we have
$k / \ell_{k,i,i} < e_{H_i} / v_{H_i} = d_{H_i} < \alpha$, and therefore $\gamma_i>0$.

When $d_{H_i} > \alpha$, we have that
\[
\mathbb{E}[W_i] 
\leq n^{v_{H_i}}(cn^{-1/\alpha})^{e_{H_i}} 
= c^{e_{H_i}}n^{v_{H_i}(1 - d_{H_i}/\alpha)}
\leq \exp(-Cn^{\eta_i}),
\]
with $\eta_i>0$.
By Markov's inequality,
\[\mathbb{P}(W_i > 0) \leq \mathbb{E}[W_i] \leq \exp(-Cn^{\eta_i}).\]

We now consider the case $d_{H_i} = \alpha$ for all $i \in I$.
First observe that for all $i \in I$, as $n \rightarrow \infty$,
$
\E(W_i) \rightarrow  c^{\alpha v_{H_i}}/a_{H_i}.
$
  We next bound the two terms in Corollary \ref{t5}. For this step, in order to simplify the exposition, we assume, without loss of generality, that $I=\{1,\ldots, d\}$.
We start bounding the first sum. We observe that $\lambda_i=O(1)$ and that $n^{v_{H_i}}p^{e_{H_i}}=O(1)$. In particular,
$p^{e_{H_i}}=p^{e_{H_i}}n^{v_{H_i}}n^{-v_{H_i}}=O(n^{-v_{H_i}})$.
Then,  recalling the definition of $\gamma$ in \eqref{gamma} and observing that $\gamma>0$ since all the graphs are strictly balanced, we obtain that for large $n$,
\begin{equation*} 
\sum_{i=1}^d  \min\{1,\lambda_i\} 
\big(p^{e_{H_i}}+n^{v_{H_i}-\gamma_i} p^{e_{H_i}}\big)
=O(n^{-\gamma}).
\end{equation*}
We next bound the second sum. By the definition of $p$, we have that for some constant $C>0$ and large $n$,
\begin{equation*} \begin{split}
\sum_{i=2}^d 
\sum_{j=1}^{i-1} \lambda_j
\sum_{k \in \mathcal{K}_{i,j}} p^{-k} n^{-\ell_{k,i,j}} 
\leq C \sum_{i=2}^d \sum_{j=1}^{i-1}\sum_{k \in \mathcal{K}_{i,j}} n^{-(\ell_{k,i,j}-k/\alpha)}.
\end{split}
\end{equation*}
Now, observe  that given $(i,j)$ since $\alpha=d_{H_i}$ and all the graphs have the same density and are strictly balanced, we have that
$
\ell_{k,i,j}-k/\alpha \geq \gamma>0
$
since if a pair $(k,\ell_{k,i,j})$ is in the sum, then the graph with $k$ edges and $\ell_{k,i,j}$ vertices is a subgraph of both $H_i$ and $H_j$. Thus, 
the second sum is also $O(n^{-\gamma})$ as $n$ is large,
which implies the desired result.
\end{proof}

\begin{remark} \label{re}
We note that it is possible to get results for subgraphs that are not strictly balanced,
as Corollary \ref{t5} does not assume `balancedness'.
In particular,
in the case where $H$ has a unique densest subgraph,
the desired result can be deduced immediately from our results. 
However, other cases are
more delicate, with different subgraphs of $H$ `competing'. 
One would need to incorporate considerations similar to 
those in the determination the threshold of appearance of $H$, as was done by Bollob\'as 
\cite[Section 4.2]{bela} to obtain approximation errors for particular graphs.
\end{remark}

We end this section with an example of application of Theorem \ref{t5b}.
\begin{example}[Vector of cycles] 
Let $d \geq 2$ and let $H_1,\ldots,H_d$ be $d$ distinct cycles, where each $H_i$ is of length $k_i \geq 3$, 
with $k_1 < k_2 < \cdots < k_d$. 
Then we have
    \begin{align*}
    \mathbb{E}(W_i)
    = \frac{n(n-1)\cdots(n-k_i+1)}{2k_i} p^{k_i}.
    \end{align*}
For any two cycles $H_i$ and $H_j$ with $i < j$, the maximum shared edges is $M_{i,j} = k_i-1$,
and for $k = 1,2,\ldots,k_i-1$, the minimum vertices for $k$ shared edges is
$\ell_{k,i,j} = k+1$.
Since all cycles have density equal to one, taking $p = cn^{-1}(1 + o(1))$, by Theorem \ref{t5b}, we obtain,
by noting that $\gamma=1$, that for $n$ sufficiently large,
\begin{equation*}
d_{\mathrm{W}}(\mathbf W, {\bf  P_{\bm\lambda}}) \leq Cn^{-1},
\end{equation*}
where $\lambda_i=\E(W_i) \rightarrow c^{k_i}/(2k_i)$ for $i = 1,\ldots,d$.
\end{example}

\subsection{Multivariate hypergeometric distribution}

In this section we give Poisson approximation of the multivariate hypergeometric distribution, 
starting with its definition.

\begin{definition}[Multivariate hypergeometric distribution]
Consider an urn containing $N$ balls of $d$ different colors, 
with $n_i$ balls of color $i$ ($i=1,\ldots,d$) such that $\sum_{i=1}^d n_i = N$. 
When drawing $m$ balls without replacement, 
let $\mathbf{W} = (W_1,\ldots,W_d)$, 
where $W_i$ denotes the number of balls of color $i$ in the sample.
Then $\mathbf{W}$ follows a multivariate hypergeometric distribution.
\end{definition}

Next, as a consequence of the decreasing size-biased coupling in Theorem \ref{dd}, we obtain the following  Poisson approximation, which is the multivariate extension of \cite[Example 4.32]{R}. Observe that in this application of Theorem \ref{dd} the $p_{i,j}$'s do not depend on $i$ nor $j$.
\begin{theorem}
Let $\mathbf{W} = (W_1,\ldots,W_d)$ follow the multivariate hypergeometric distribution defined as above.
Then
\begin{align*} 
d_{\mathrm{W}}(\mathbf{W},{\bf  P_{\bm\lambda}}) 
&\le \sum_{i=1}^d \min\left\{1,\frac{mn_i}{N}\right\} 
\left(1 - \frac{(N-n_i)(N-m)}{N(N-1)} \right)
+ \frac{2(N-m)}{N(N-1)}
\sum_{i=2}^d \sum_{j=1}^{i-1} n_j,
\end{align*}
where $\lambda_i = \E( W_i ) = mn_i/N$ for each $i=1,\ldots,d$.
\end{theorem}

\begin{remark}
As in  the one-dimensional case, this approximation is useful when $N$ is large 
and $m/N$ and $n_i/N$ are small for all $i=1,\ldots,d$.
\end{remark}

\begin{proof}
For each color $i \in \{1,\ldots,d\}$, note that $W_i$ can be written as a sum of indicator random variables
$
W_i = \sum_{j=1}^{n_i} X^i_j,
$
where $X^i_j$ is the indicator that ball $j$ of color $i$ is chosen in the $m$-sample
for $j = 1, \ldots, n_i$.
So $\P(X_j^i=1) = m/N$.
For $\mathbf{W} = (W_1,\ldots,W_d)$,
we construct its size-biased coupling 
${\bf  \widetilde W}=({\bf \widetilde W}^{1}, \ldots, {\bf \widetilde W}^{d})$
with the random vector $\mathbf{\widetilde W}^{i}$ defined as follows for each $i \in \{1,\ldots,d\}$.
Let $I_i$ be uniformly distributed on $\{1,\ldots,n_i\}$, independent of all else. Then, if ball $I_i$ of color $i$ is already in the $m$-sample, 
let $\mathbf{\widetilde W}^{i} = \mathbf{W}^{i}$; otherwise,
we add ball $I_i$ of color $i$ to the sample,
and remove a uniformly chosen ball from the current sample.
Recall that the law of coupling from \eqref{law}
specifies that the distribution of the coupled random vector matches 
the conditional distribution of the original vector given that ball $\ell$ of color $i$ was selected.
Then we have, for all $i=1, \ldots, d$, that
we can write the ${\bf \widetilde W}^{i}$ defined above as
\begin{align*}
{\bf \widetilde W}^{i}
= \bR{ \sum_{j=1}^{n_1} X_{j}^{1,(i,I_i)}, \ldots, \sum_{j=1}^{n_{i-1}} X_{j}^{i-1,(i,I_i)},
\sum_{j=1, j \neq I_i}^{n_i} X_{j}^{i,(i,I_i)}+1 },
\end{align*}
where for colors $k < i$,
$\sum_{j=1}^{n_k} X_j^{k,(i,I_i)}$ represents the count of color $k$ balls after coupling, 
and for color $i$, 
$\sum_{j=1,j\neq I_i}^{n_i} X_j^{i,(i,I_i)} + 1$ represents the count including the mandatory ball $I_i$.
Then $X_{k}^{j,(i,\ell)}\leq  X_{k}^{j}$ for $j=1, \ldots, i-1$, $\ell=1, \ldots, n_i$, and $k=1, \ldots, n_j$,
and for $j=i$ and $k \neq \ell$ we have $X_{k}^{i,(i,\ell)}\leq  X_{k}^{j}$,
and therefore, $\widetilde W^{i}_j \leq W_j$ for all $j \neq i$, 
as no additional balls of any color other than $i$ can be added during the coupling procedure. 
Thus, we can apply Theorem \ref{dd}.
Standard hypergeometric calculations yield
\begin{align}
\label{lam}
\mathbb{E}[W_i] = \frac{mn_i}{N} = \lambda_i,
\text{ and }
\text{Var}(W_i) = \frac{mn_i(N-n_i)(N-m)}{N^2(N-1)} = \lambda_i\frac{(N-n_i)(N-m)}{N(N-1)}.
\end{align}
For the covariance between counts with distinct colors $i \ne j$, similar calculations give
\[
\text{Cov}(W_i,W_j) = -\frac{mn_in_j(N-m)}{N^2(N-1)} = -\lambda_i\lambda_j\frac{N-m}{m(N-1)}.
\]
Hence by Theorem \ref{dd}, we have
\begin{align*}
d_{\mathrm{W}}({\bf W}, {\bf  P_{\bm\lambda}}) 
&\leq \sum_{i=1}^d \min\left\{1,\lambda_i\right\} 
\left(1 - \frac{(N-n_i)(N-m)}{N(N-1)} \right)
+ \frac{2(N-m)}{m(N-1)}
\sum_{i=2}^d \sum_{j=1}^{i-1} \lambda_j.
\end{align*}
Then the proof is complete
in view of the expectation formula in \eqref{lam}.
\end{proof}

\section*{Acknowledgement}

This project was initiated at the Probability and Combinatorics Workshop held at
the Bellairs Research Institute in March 2024. 
The authors thank the organisers, G\'abor Lugosi, Louigi Addario-Berry, and Nicolas Broutin. 
Special thanks go to G\'abor Lugosi for useful comments, suggestions, and discussions. 
The first author acknowledges support from the Spanish MINECO grant PID2022-138268NB-100 and
Ayudas Fundacion BBVA a Proyectos de Investigaci\'on Cient\'ifica 2021.

\end{document}